\documentclass{amsart}
\usepackage{amssymb}
\usepackage{mathrsfs}
\usepackage[all]{xy}

\newcommand{\bF}{\mathbb{F}}

\newcommand{\cD}{\mathcal{D}}
\newcommand{\Z}{\mathbb{Z}}
\newcommand{\C}{\mathbb{C}}

\newcommand{\Qlb}{\bar{\mathbb{Q}}_\ell}

\newcommand{\Db}{D^{\mathrm{b}}}
\newcommand{\Cb}{C^{\mathrm{b}}}
\newcommand{\Kb}{K^{\mathrm{b}}}
\newcommand{\Tilt}{\mathrm{Tilt}}
\newcommand{\Proj}{\mathrm{Proj}}
\newcommand{\real}{\mathrm{real}}

\newcommand{\Perv}{\mathrm{Perv}}
\newcommand{\IC}{\mathbf{IC}}
\newcommand{\Ext}{\mathrm{Ext}}
\newcommand{\Hom}{\mathrm{Hom}}

\newcommand{\cN}{\mathcal{N}}
\newcommand{\fg}{\mathfrak{g}}
\newcommand{\fz}{\mathfrak{z}}
\newcommand{\fp}{\mathfrak{p}}
\newcommand{\ft}{\mathfrak{t}}
\newcommand{\fl}{\mathfrak{l}}
\newcommand{\fu}{\mathfrak{u}}

\newcommand{\cW}{\mathcal{W}}

\newcommand{\uF}{\underline{\mathbb{F}}}

\newcommand{\cR}{\mathscr{R}}
\newcommand{\cRp}{\:{}^\backprime\mathscr{R}}
\newcommand{\bT}{\mathbb{T}}
\newcommand{\bTp}{\:{}^\backprime\mathbb{T}}
\newcommand{\D}{\mathbb{D}}
\newcommand{\id}{\mathrm{id}}
\newcommand{\cK}{\mathcal{K}}
\newcommand{\pH}{{}^p\! H}
\newcommand{\cE}{\mathcal{E}}

\newcommand{\cT}{\mathcal{T}}
\newcommand{\cP}{\mathcal{P}}

\newcommand{\cO}{\mathcal{O}}
\newcommand{\cC}{\mathcal{C}}
\newcommand{\cL}{\mathcal{L}}

\def\e{\varepsilon}

\newcommand{\ul}{\underline}

\def\l{\lambda}
\def\d{\delta}

\newcommand{\Spr}{{\mathbf{Spr}}}

\DeclareMathOperator{\res}{res}
\newcommand{\resp}{\mathop{{}^\backprime\mathrm{res}}\nolimits}
\DeclareMathOperator{\ind}{ind}
\newcommand{\con}{\mathrm{con}}
\newcommand{\reg}{\mathrm{reg}}
\newcommand{\sub}{\mathrm{sub}}
\newcommand{\midd}{\mathrm{mid}}
\newcommand{\simto}{\overset{\sim}{\to}}
\newcommand{\Irr}{\mathrm{Irr}}
\newcommand{\End}{\mathrm{End}}
\newcommand{\Cent}{\mathrm{Z}}
\newcommand{\Cento}{\mathrm{Z}^\circ}

\newcommand{\indLG}{\ind_L^G}
\newcommand{\resLG}{\res_L^G}
\newcommand{\respLG}{\resp_L^G}
\newcommand{\indLGl}{\ind_{L_\lambda}^G}

\numberwithin{equation}{section}
\newtheorem{thm}{Theorem}[section]
\newtheorem{lem}[thm]{Lemma}
\newtheorem{prop}[thm]{Proposition}
\newtheorem{cor}[thm]{Corollary}
\newtheorem{conj}[thm]{Conjecture}
\theoremstyle{definition}
\newtheorem{defn}[thm]{Definition}
\theoremstyle{remark}
\newtheorem{rmk}[thm]{Remark}

\title[Nilpotent cones, Fourier transform, geometric Ringel duality]{Sheaves on nilpotent cones, Fourier transform, and a geometric Ringel duality}

\author{Pramod N. Achar}
\address{Department of Mathematics\\
  Louisiana State University\\
  Baton Rouge, LA \ 70803\\
  U.S.A.}
\email{pramod@math.lsu.edu}

\author{Carl Mautner}
\address{Department of Mathematics\\
  FAS\\
  Harvard University\\
  One Oxford Street\\
  Cambridge, MA \ 02138\\
  U.S.A.}
\email{cmautner@math.harvard.edu}

\thanks{The first author was supported by NSF Grant No.~DMS-1001594 and the second author was supported by an NSF postdoctoral research fellowship.}

\begin{document}

\begin{abstract}
Given the nilpotent cone of a complex reductive Lie algebra, we consider its equivariant constructible derived category of sheaves with coefficients in an arbitrary field.  This category and its subcategory of perverse sheaves play an important role in Springer theory and the theory of character sheaves.  We show that the composition of the Fourier--Sato transform on the Lie algebra followed by restriction to the nilpotent cone restricts to an autoequivalence of the derived category of the nilpotent cone.  In the case of $GL_n$, we show that this autoequivalence can be regarded as a geometric version of Ringel duality for the Schur algebra.
\end{abstract}

\maketitle

%%%%%%%%%%%%%%%%%%%%%%%%%%%%%%%%%%%%%%%%%%%%%%%%%%%%%%%%%%%%%%%%%%%%%%%%%%%
\section{Introduction}
%%%%%%%%%%%%%%%%%%%%%%%%%%%%%%%%%%%%%%%%%%%%%%%%%%%%%%%%%%%%%%%%%%%%%%%%%%%

Let $G$ be a connected complex reductive group, and let $\cN$ be the variety of nilpotent elements in its Lie algebra $\fg$.  Let $\bF$ be a field.  The role of Fourier transforms in the study of perverse $\bF$-sheaves on $\fg$ or on $\cN$ is well-established: for $\bF = \Qlb$ or $\C$, see~\cite{hk,lus:ftsla,mirk}, and for $\bF$ of positive characteristic, see~\cite{juteau}.  Much of the literature implicitly considers the functor
\[
\cR: \cD_G(\cN,\bF) \to \cD_G(\cN,\bF)
\]
given by Fourier transform on $\fg$ followed by restriction to $\cN$ (and a suitable shift; see~\eqref{eqn:cr-defn}).    When $\bF = \C$, in the context of the Springer correspondence, $\cR$ corresponds to tensoring by the sign character of the Weyl group.  In the present paper, we show that $\cR$ is actually an autoequivalence for arbitrary $\bF$.

In addition, when $G = GL_n$, we show that $\cR$ is a geometric version of Ringel duality.  More precisely: the second author has shown~\cite{mau} that the abelian category of perverse sheaves $\Perv_G(\cN,\bF)$ for $G = GL_n$ is equivalent to the category $S_\bF(n,n)\textup{-mod}$ of finitely-generated modules for the Schur algebra $S_\bF(n,n)$ over $\bF$.  Donkin has shown~\cite[Proposition 3.7]{do} that the Schur algebra is Ringel self-dual; this means that there is an autoequivalence of the derived category $\Db (S_\bF(n,n)\textup{-mod})$ that sends tilting modules to projective modules.  The second main result of the present paper asserts that for $G = GL_n$, $\cR$ has a similar property: it sends tilting perverse sheaves to projective perverse sheaves.

Recall that the category $\cD_G(\cN_G,\bF)$ is \emph{not} equivalent to the derived category $\Db\Perv_G(\cN_G,\bF)$, so the functor $\cR$ is not the usual (``algebraic'') Ringel duality of derived categories of abelian categories.  Nevertheless, we will see in Section~\ref{sect:ringel} that $\cR$ does give rise to an algebraic Ringel duality functor, and that the two are compatible in a suitable sense.

Indeed, we obtain in this way a new proof of Donkin's result.  It is perhaps interesting to note the difference in perspective: in~\cite{do}, the focus is on the structure of the endomorphism rings of a tilting generator $\cT$ and a projective generator $\cP$, and the construction of an explicit isomorphism  between them.  On the other hand, in the present paper, by starting with the geometric autoequivalence $\cR$, we immediately obtain an isomorphism $\End(\cT) \simto \End(\cR(\cT))$; the focus is then on showing that $\cR(\cT)$ is a projective generator.  This approach avoids explicit descriptions of endomorphism rings.

The paper is organized as follows.  After some general preliminaries in Section~\ref{sect:prelim}, we review parabolic induction and restriction functors for perverse sheaves in Section~\ref{sect:resind}.  The fact that $\cR$ is an autoequivalence is proved in Section~\ref{sect:autoequiv}, and the relationship with Ringel duality in the case of $G = GL_n$ is studied in Section~\ref{sect:ringel}.  Lastly, Section~\ref{sect:behavior} contains a (partly conjectural) description of the functor in the case $\bF = \C$ and some observations regarding the behavior of $\cR$ when $G$ is of type $B_2$ or $G_2$.

%--------------------------------------------------------------------------
\subsection*{Acknowledgments}
%--------------------------------------------------------------------------

The authors would like to express their gratitude to A.~Henderson, K.~McGerty, I.~Mir\-kovi\'c, and S.~Riche for helpful comments on a previous version of this paper.

%%%%%%%%%%%%%%%%%%%%%%%%%%%%%%%%%%%%%%%%%%%%%%%%%%%%%%%%%%%%%%%%%%%%%%%%%%%
\section{Preliminaries}
\label{sect:prelim}
%%%%%%%%%%%%%%%%%%%%%%%%%%%%%%%%%%%%%%%%%%%%%%%%%%%%%%%%%%%%%%%%%%%%%%%%%%%

%--------------------------------------------------------------------------
\subsection{Varieties and sheaves}
%--------------------------------------------------------------------------

In this paper, varieties are generally assumed to be over $\C$ and equipped with the strong topology.  If $X$ is a variety acted on by an algebraic group $H$, we write $\cD_H(X,\bF)$ for the bounded $H$-equivariant constructible derived category of $X$ of sheaves of $\bF$-vector spaces, in the sense of~\cite{bl}.  Given a closed subgroup $K \subset H$, there is a functor (called the ``integration functor'')
\[
\Gamma_K^H: \cD_K(X,\bF) \to \cD_H(X,\bF)
\]
that is right adjoint to the forgetful functor $\cD_H(X,\bF) \to \cD_K(X,\bF)$. (See~\cite[Theorem~3.7.1]{bl}, where this functor is denoted $\mathrm{Ind}_*$.)  

Let $\D : \cD_H(X,\bF) \to \cD_H(X,\bF)$ denote the Verdier duality functor, and let $\Perv_H(X,\bF) \subset \cD_H(X,\bF)$ be the abelian category of $H$-equivariant perverse sheaves.  We also have perverse cohomology functors $\pH^i: \cD_H(X,\bF) \to \Perv_H(X,\bF)$ for each $i \in \Z$.  For $M, N \in \Perv_H(X,\bF)$, there is a natural isomorphism
\[
\Ext^1_{\Perv_H(X,\bF)}(M,N) \cong \Hom_{\cD_H(X,\bF)}(M, N[1]).
\]
This fact will be used several times in the sequel.

We write $\uF_X$, or simply $\uF$, for the constant sheaf of value $\bF$ on $X$.  Next, let $j_C: C \hookrightarrow X$ be the inclusion of an $H$-stable subset, and let $E$ be an irreducible $H$-equivariant local system on $C$. In addition to the simple perverse sheaf $\IC(C,E) = (j_C)_{!*}(E[\dim C])$, we will sometimes consider the perverse sheaves
\[
\Delta(C,E) = \pH^0(j_{C!}E[\dim C])
\qquad\text{and}\qquad
\nabla(C,E) = \pH^0(j_{C*}E[\dim C]),
\]
often called \emph{standard} and \emph{costandard} perverse sheaves, respectively.

%--------------------------------------------------------------------------
\subsection{Vector spaces and conic sheaves}
%--------------------------------------------------------------------------

Let $V$ be a complex vector space acted on linearly by the group $H$.  We say that an object $M \in \cD_H(V,\bF)$ is \emph{conic} if for each nonzero vector $v \in V$ and each $i \in \Z$, the sheaf $H^i(M)|_{\C\cdot v \smallsetminus \{0\}}$ on $\C\cdot v \smallsetminus \{0\}$ is locally constant.  In other words, $M$ is conic if it is constructible with respect to the obvious scaling action of $\C^\times$ on $V$.  Let $\cD_{H,\con}(V,\bF) \subset \cD_H(V,\bF)$ denote the full subcategory of conic objects.

Let $\delta_V$ denote the skyscraper sheaf at the origin of $V$ (with value $\bF$).  It is clear that $\delta_V$ is conic.

%--------------------------------------------------------------------------
\subsection{Fourier--Sato transform}
%--------------------------------------------------------------------------

Let $V$ be as above, and let $V^*$ denote the dual vector space.  In this paper, we adopt the convention that the \emph{Fourier--Sato transform} for $V$, denoted
\[
\bT_V: \cD_{H,\con}(V,\bF) \to \cD_{H,\con}(V^*,\bF),
\]
is defined by composing the functor defined in~\cite[\S 3.7]{ks} with the shift $[\dim V]$.  With this modification, $\bT_V$ is $t$-exact for the perverse $t$-structure~\cite[Proposition~10.3.18]{ks}.  This functor is an equivalence of categories; its inverse is denoted
\[
\bTp_V: \cD_{H,\con}(V^*,\bF) \to \cD_{H,\con}(V,\bF).
\]
We will need the following key properties of $\bT_V$ (see~\cite[\S 3.7]{ks}):
\begin{enumerate}
\item $\bT_V(\delta_V) \cong \uF_{V^*}[\dim V^*]$.
\item If $V = V_1 \times V_2$, then $\bT_V(M_1 \boxtimes M_2) \cong \bT_{V_1}(M_1) \boxtimes \bT_{V_2}(M_2)$.
\end{enumerate}

%--------------------------------------------------------------------------
\subsection{Highest-weight categories and tilting}
\label{subsect:hwc}
%--------------------------------------------------------------------------

The notion of \emph{highest weight category}, due to Cline, Parshall, and Scott \cite{cps}, involves the following data:
\begin{enumerate}
\item a noetherian and artinian $\bF$-linear abelian category $\cC$ with enough projectives and injectives
\item a partially ordered set $(I, \le)$ that indexes a set $\{ L_i \mid i \in I\}$ of representatives of the isomorphism classes of simple objects in $\cC$
\item for each $i \in I$, a specified \emph{standard} object $\Delta_i$ and a \emph{costandard} object $\nabla_i$.
\end{enumerate}
The standard and costandard objects are required to satisfy various additional properties whose details we omit.  In this paper, we assume that $I$ is finite.

An object is said to be \emph{tilting} if it has both a filtration by standard objects and a filtration by costandard objects.  The indecomposable tilting objects are again indexed by $I$: say $\{T_i \mid i \in I\}$.

Suppose now that $\cC'$ is another highest weight category, with simple objects $\{L'_i \mid i \in I \}$ indexed by $(I, \le^{\mathrm{op}})$ (that is, $I$ with the opposite partial order).  Let $P'_i$ be a projective cover of $L'_i$ in $\cC'$.  Then $\cC'$ is said to be \emph{Ringel dual} to $\cC$ if there is an isomorphism
\begin{equation}\label{eqn:ringel-isom}
\vartheta: \End\Big(\bigoplus_{i \in I} T_i\Big) \simto \End\Big(\bigoplus_{i \in I} P'_i\Big).
\end{equation}
Ringel duals always exist~\cite{ringel}.  Each isomorphism as in~\eqref{eqn:ringel-isom} gives rise to a derived equivalence $H_\vartheta: \Db\cC \simto \Db\cC'$.  More specifically, let $\Tilt(\cC)$ (resp.~$\Proj(\cC')$) be the additive category of tilting objects in $\cC$ (resp.~projective objects in $\cC'$).  Then $\vartheta$ induces an equivalence of additive categories
\begin{equation}\label{eqn:ringel-add}
H_\vartheta: \Tilt(\cC) \to \Proj(\cC'),
\end{equation}
since both are naturally equivalent to the category of projective modules over the ring in~\eqref{eqn:ringel-isom}.  This extends to an equivalence of homotopy categories $H_\vartheta: \Kb\Tilt(\cC) \to \Kb\Proj(\cC')$.  Finally, using the natural equivalences
\begin{equation}\label{eqn:tiltproj}
\Kb\Tilt(\cC) \cong \Db\cC
\qquad\text{and}\qquad
\Kb\Proj(\cC') \cong \Db\cC'
\end{equation}
(see~\cite[Lemmas~1.1 and~1.5]{happel} and~\cite[Theorem~4.3]{ps}), we obtain a derived equivalence $H_\vartheta: \Db\cC \simto \Db\cC'$.

\begin{rmk}\label{rmk:etale}
Instead of working with complex reductive groups in the strong topology and Fourier--Sato transform, we could work with reductive groups over an algebraically closed field of characteristic $p > 0$ in the \'etale topology, and use the Fourier--Deligne transform (see~\cite{laumon} for $\bF= \Qlb$; in the modular case, the relevant theory is developed in~\cite{juteau}).  Note that $\bF$ cannot be completely arbitrary in this case: it must have characteristic different from $p$, and it is required to contain the $p$-th roots of unity.  The main arguments of the paper apply unchanged in this setting; in particular, analogues of Theorems~\ref{thm:main} and~\ref{thm:ringel} are still true.
%The arguments of Section~\ref{sect:autoequiv} apply unchanged in this setting; in particular, an analogue of Theorem~\ref{thm:main} is still true.  The results of Section~\ref{sect:ringel}, however, depend on those in~\cite{mau}, which are specific to $GL_n(\C)$.
\end{rmk}

%%%%%%%%%%%%%%%%%%%%%%%%%%%%%%%%%%%%%%%%%%%%%%%%%%%%%%%%%%%%%%%%%%%%%%%%%%%
\section{Restriction and induction}
\label{sect:resind}
%%%%%%%%%%%%%%%%%%%%%%%%%%%%%%%%%%%%%%%%%%%%%%%%%%%%%%%%%%%%%%%%%%%%%%%%%%%

Recall that $G$ denotes a connected complex reductive group, $\fg$ its Lie algebra, and $\cN_G$ its nilpotent cone.  We fix, once and for all, a Borel subgroup $B \subset G$ and a maximal torus $T \subset B$.

Let $P$ be a parabolic subgroup of $G$ containing $B$, and let $L \subset P$ be the Levi factor containing $T$.  Their Lie algebras will be denoted $\fp$ and $\fl$, respectively.  Let $\cN_L$ be the nilpotent cone in $\fl$.  If we let $P$ act on $\fl$ through the quotient map $P \twoheadrightarrow L$, then, by~\cite[Theorem~3.7.3]{bl}, the forgetful functor
\[
\cD_P(\fl,\bF) \to \cD_L(\fl,\bF)
\]
is an equivalence of categories.  We will henceforth suppress mentions of this equivalence, and silently pass between the two.  For instance, consider the diagam
\[
\xymatrix{ \fg & \fp \ar[l]_-{m} \ar[r]^-{p} & \fl },
\]
where $m: \fp \to \fg$ is the inclusion map, and $p: \fp \to \fl$ is the projection map.  We define two functors
\[
\resLG, \respLG: \cD_G(\fg,\bF) \to \cD_L(\fl,\bF)
\qquad\text{by}\qquad
\resLG = p_!m^*, \qquad \respLG = p_*m^!.
\]
and
\[
\indLG: \cD_L(\fl,\bF) \to \cD_G(\fg,\bF)
\qquad\text{by}\qquad
\indLG = \Gamma_P^G m_* p^!.
\]

It is easy to see that if $M \in \cD_G(\fg,\bF)$ is supported on the nilpotent cone $\cN_G$, then $\resLG M$ and $\respLG M$ are supported on $\cN_L$.  Similarly, $\indLG$ takes objects supported on $\cN_L$ to objects supported on $\cN_G$.  Thus, if we let
\[
i_L: \cN_L \to \fl
\qquad\text{and}\qquad
i_G: \cN_G \to \fg
\]
be the inclusion maps, we see that there are functors (unique up to isomorphism)
\[
\resLG, \respLG: \cD_G(\cN_G,\bF) \to \cD_L(\cN_L,\bF),
\qquad
\indLG: \cD_L(\cN_L,\bF) \to \cD_G(\cN_G,\bF)
\]
satisfying
\begin{equation}\label{eqn:ipush-comm}
i_{L*} \resLG \cong \resLG i_{G*}, \qquad
i_{L*} \respLG \cong \respLG i_{G*}, \qquad
i_{G*} \indLG \cong \indLG i_{L*}.
\end{equation}
It is immediate from the definitions that
\begin{equation}\label{eqn:res-duality}
\D \circ \resLG \cong \respLG \circ \D.
\end{equation}
The following proposition collects some additional properties of these functors that are well known when $\bF = \C$.  For arbitrary $\bF$, related results appear in~\cite[\S 4]{ahr}.

\begin{prop} \label{prop:t-exact}
\begin{enumerate}
\item The functor $\indLG$ commutes with $\D$.\label{it:verdier}
\item The functor $\indLG$ is right adjoint to $\resLG$ and left adjoint to $\respLG$.\label{it:iradjoint}
\item The functors $\resLG$, $\respLG$, and $\indLG$ are $t$-exact for the perverse $t$-structures on $\cD_G(\cN_G,\bF)$ and $\cD_L(\cN_L,\bF)$.\label{it:irexact}
\end{enumerate}
\end{prop}

\begin{proof}
\eqref{it:verdier}~Let $\nu: \cD_P(\fg) \to \cD_G(G \times^P \fg)$ be the ``induction equivalence'' \cite[Theorem~2.6.3]{bl}, and let $s: G \times^P \fg \to \fg$ be the map $s(g,x) = g \cdot x$.  By definition, on $\fg$, we have $\Gamma_P^G = s_*Q$, so $\indLG = s_*Qm_*p^!$.  Since $s$ and $m$ are proper, $s_*$ and $m_*$ commute with $\D$.  Since $p$ is smooth of relative dimension $\dim \fp - \dim \fl = \dim G/P$, we have $\D \circ p^!\circ \D \cong p^![-2\dim G/P]$.  Finally, $\D \circ Q \circ \D \cong Q[2\dim G/P]$ by \cite[Proposition~7.6.2]{bl}.  Thus, $\indLG$ commutes with $\D$.

\eqref{it:iradjoint}~The first assertion is obvious, and the second follows from~\eqref{eqn:res-duality} and part~\eqref{it:verdier}.

\eqref{it:irexact}~The result for $\respLG$ is in~\cite[Proposition~4.7]{ahr}, and by~\eqref{eqn:res-duality}, $\resLG$ is $t$-exact as well.  Lastly, $\indLG$ is $t$-exact because it has $t$-exact adjoints on both sides.
\end{proof}

\begin{defn}\label{defn:cuspidal}
A simple perverse sheaf $M \in \Perv_G(\cN_G,\bF)$ is said to be \emph{cuspidal} if $\respLG M = 0$ for all Levi subgroups $L \subsetneq G$.
\end{defn}

It can be shown that this is equivalent to requiring that $\resLG M = 0$ for all $L \subsetneq G$, using for instance,~\cite[Theorem~4.7]{mirk} or~\cite[Theorem~12]{lus:ftsla}.  We will not require this fact.

\begin{lem}\label{lem:induced}
Let $M \in \Perv_G(\cN_G,\bF)$ be a simple perverse sheaf.  If $M$ is not cuspidal, then $M$ is a quotient of some induced perverse sheaf.
\end{lem}
\begin{proof}
Let $\fl \subsetneq \fg$ be a Levi subalgebra such that $\respLG M \ne 0$, and let $N = \indLG \respLG M$.  By adjunction, we have a natural nonzero map $N \to M$.  Since $M$ is simple, this map is surjective.
\end{proof}

\begin{lem}\label{lem:ires-comm}
The restriction and induction functors commute with $i^*$ and $i^!$:
\begin{align*}
i_L^* \resLG &\cong \resLG i_G^*, &
i_L^* \respLG &\cong \respLG i_G^*, &
i_G^* \indLG &\cong \indLG i_L^*, \\
i_L^! \resLG &\cong \resLG i_G^!, &
i_L^! \respLG &\cong \respLG i_G^!, &
i_G^! \indLG &\cong \indLG i_L^!.
\end{align*}
\end{lem}
\begin{proof}
Let $\cN_P = \cN_G \cap \fp$, and consider the following commutative diagram:
\[
\xymatrix{ 
\cN_G \ar[d]_{i_G} & \cN_P \ar[d]_{i_P} \ar[l]_-{m} \ar[r]^-{p} & \cN_L \ar[d]^{i_L} \\
\fg & \fp \ar[l]_-{m} \ar[r]^-{p} & \fl }
\]
Both squares are cartesian.  In particular, we have $p_!i_P^* \cong i_L^* p_!$, and from this it follows that $\resLG$ commutes with $i^*$.  Similarly, $\respLG$ commutes with $i^!$.

Next, let $\bar P$ be the opposite parabolic to $P$, and let $\bar\fp$ be its Lie algebra.  We can form a diagram similar to the one above:
\[
\xymatrix{ 
\cN_G \ar[d]_{i_G} & \cN_{\bar P} \ar[d]_{i_{\bar P}} \ar[l]_-{\bar m} \ar[r]^-{\bar p} & \cN_L \ar[d]^{i_L} \\
\fg & \bar\fp \ar[l]_-{\bar m} \ar[r]^-{\bar p} & \fl }
\]
As in the previous paragraph, we deduce from this diagram that the functor $\bar p_* \bar m^!$ commutes with $i^!$.  But the main result of~\cite{bra} implies that $\resLG \cong \bar p_*\bar m^!$, so $\resLG$ commutes with $i^!$.  Similarly, the functor $\respLG \cong \bar p_! \bar m^*$ commutes with $i^*$.

Finally, $i^* \indLG$ and $\indLG i^*$ are left adjoint to $\respLG i_*$ and $i_*\respLG$, respectively.  In view of~\eqref{eqn:ipush-comm} and the fact that adjoints are unique, we conclude that $i_G^*\indLG \cong \indLG i_L^*$.  A similar argument with $\resLG$ shows that $i_G^!\indLG \cong \indLG i_L^!$.
\end{proof}

We conclude this section with a lemma on standard and costandard objects.  

\begin{lem}\label{lem:ind-std}
Consider a nilpotent orbit $j_C: C \hookrightarrow \cN_L$ and a local system $E$ on $C$.
\begin{enumerate}
\item Every simple subobject $\IC(C',E')$ of $\indLG \nabla(C,E)$ satisfies $G \cdot C \subset \overline{C'}$.
\item Every simple quotient $\IC(C',E')$ of $\indLG \Delta(C,E)$ satisfies $G \cdot C \subset \overline{C'}$.
\end{enumerate}
\end{lem}
\begin{proof}
Let $C_1 = p^{-1}(C) \subset \fp$, and let $V = G \cdot C_1$.  Let $h: V \to \cN_G$ be the inclusion map.  It can be seen from the proof of~\cite[Theorem~3.7.1]{bl} (using~\cite[Theorem~3.4.1]{bl}) that the integration functor $\Gamma_P^G$ commutes with $h_*$.  As a consequence, if we let $F = \Gamma_P^G (m|_{C_1})_* (p|_{C_1})^! E$, it is easy to deduce that $\indLG j_{C*}E[\dim C] \cong h_*F$.  Since $\indLG$ is $t$-exact, it follows that $\indLG \nabla(C,E) \cong \pH^0(h_*F)$, and the latter has no simple subobject $\IC(C',E')$ with $C' \not\subset V$.  Since $G \cdot C$ is the unique minimal $G$-orbit contained in $V$, the first assertion in the lemma holds.  The second assertion then follows by Verdier duality.
\end{proof}

%%%%%%%%%%%%%%%%%%%%%%%%%%%%%%%%%%%%%%%%%%%%%%%%%%%%%%%%%%%%%%%%%%%%%%%%%%%
\section{Construction of the autoequivalence}
\label{sect:autoequiv}
%%%%%%%%%%%%%%%%%%%%%%%%%%%%%%%%%%%%%%%%%%%%%%%%%%%%%%%%%%%%%%%%%%%%%%%%%%%

In this section, we will define the functor $\cR$ and prove one of the main results of the paper.  Let us fix, once and for all, a $G$-equivariant isomorphism $\fg \simto \fg^*$.  This allows to regard the Fourier--Sato transform $\bT_\fg$ as a functor $\cD_{G,\con}(\fg,\bF) \to \cD_{G,\con}(\fg,\bF)$.  Recall that $i: \cN \to \fg$ denotes the inclusion map.  Let
\[
\cR, \cRp: \cD_G(\cN_G,\bF) \to \cD_G(\cN_G,\bF)
\]
denote the functors given by
\begin{equation}\label{eqn:cr-defn}
\cR = i^* \bT_\fg i_*[\dim \cN_G - \dim \fg], \qquad
\cRp = i^! \bTp_\fg i_*[\dim \fg - \dim \cN_G].
\end{equation}
Note that for these definitions to make sense, we must verify that $i_*: \cD_G(\cN_G,\bF) \to \cD_G(\fg,\bF)$ actually takes values in $\cD_{G,\con}(\fg,\bF)$.  But this is immediate from the well-known fact that every nilpotent orbit in $\cN_G$ is stable under the scaling action of $\C^\times$ on $\fg$.

Below, we will make use of a number of facts about Fourier transforms of perverse sheaves on $\fg$ that are originally due to Lusztig~\cite{lus:ftsla} and were later revisited by Mirkovi\'c~\cite{mirk}.  The latter is a more suitable reference for us, because it is relatively straightforward to see that Mirkovi\'c's arguments are independent of the coefficient field $\bF$.

\begin{lem}\label{lem:r-basic}
The functor $\cR$ is left adjoint to $\cRp$.  Both functors commute with all restriction and induction functors.
\end{lem}
\begin{proof}
The first assertion is immediate from the definition.  According to~\cite[Lemma~4.2]{mirk}, Fourier--Sato transform commutes with restriction and induction, so the second assertion above follows using~\eqref{eqn:ipush-comm} and Lemma~\ref{lem:ires-comm}.
\end{proof}

\begin{thm}\label{thm:main}
The functor $\cR: \cD_G(\cN_G,\bF) \to \cD_G(\cN_G,\bF)$ is an autoequivalence, and $\cRp$ is its inverse.
\end{thm}
\begin{proof}
By Lemma~\ref{lem:r-basic}, we have adjunction maps $\id \to \cRp \cR$ and $\cR \cRp \to \id$.  We will show that the first of these is an isomorphism of functors.  A similar argument shows that the second is as well, and the theorem then follows.

To show that $\id \to \cRp\cR$ is an isomorphism, it suffices to show that it is an isomorphism on some set of objects that generate $\cD_G(\cN_G,\bF)$ as a triangulated category.  For instance, it suffices to show that it is an isomorphism for all simple perverse sheaves.

We proceed by induction on the semisimple rank of $\fg$.  If $\fg$ is abelian, then $\cN_G$ is a point.  The constant sheaf $\uF$ is the unique simple perverse sheaf on $\cN_G$.  It is easy to see by direct computation that $\cR(\uF) \cong \uF$ and $\cRp(\uF) \cong \uF$.  The adjunction morphism $\uF \to \cRp(\cR(\uF))$ is nonzero, and hence an isomorphism.

Now, assume that $\fg$ has semisimple rank at least~$1$. Let $\fg'$ be its derived subalgebra.  We claim that if the theorem holds for $\fg'$, then it holds for $\fg$.  Let $\fz$ be the center of $\fg$, so that $\fg \cong \fg' \times \fz$.  Let $i_1: \cN \to \fg'$ be the inclusion map, and consider the skyscraper sheaf $\delta_\fz$.  For $M \in \cD_G(\cN_G,\bF)$, we have
\begin{multline*}
\cR_\fg(M) = i^*\bT_\fg i_*M \cong i^*(\bT_{\fg'} \boxtimes \bT_\fz)(i_{1*}M \boxtimes \delta_\fz) \\
\cong i^*(\bT_{\fg'}i_{1*}M \boxtimes \uF_\fz[\dim \fz]) \cong \cR_{\fg'}(M)[\dim \fz].
\end{multline*}
Similar reasoning shows that $\cRp_\fg(M) \cong \cRp_{\fg'}(M)[-\dim \fz]$.  From this, we deduce that $\id \to \cRp_\fg\cR_\fg$ is an isomorphism if and only if $\id \to \cRp_{\fg'} \cR_{\fg'}$ is.

We henceforth assume that $\fg$ is semisimple.  Let $j: \fg \smallsetminus \cN_G \to \fg$ be the inclusion map, and consider the functorial distinguished triangle $j_!j^* \to \id \to i_*i^* \to$.  Composing on the left with $i^! \bTp_\fg$ and on the right with $\bT_\fg i_*$, we get a functorial distinguished triangle
\begin{equation}\label{eqn:func-dt}
i^! \bTp_\fg j_!j^* \bT_\fg i_* \to \id \to \cRp \cR \to.
\end{equation}
Let $\cK = i^! \bTp_\fg j_!j^* \bT_\fg i_*$.  Showing that $\id \to \cRp \cR$ is an isomorphism is equivalent to showing that $\cK$ is the zero functor. 

For a Levi subgroup $L \subset G$, let $\cR_L$, $\cRp_L$, and $\cK_L$ denote the analogous functors on $\cD_L(\cN_L)$.  By induction, $\cK_L = 0$.  We now proceed in three steps.

{\it Step 1. If $M \in \Perv_G(\cN_G,\bF)$ is cuspidal, then $\cK(M) = 0$.} According to~\cite[Lemma~4.4]{mirk}, every cuspidal perverse sheaf on $\fg$ is in fact supported on $\cN_G$.  (Here, we are using the assumption that $\fg$ is semisimple.)  Since the Fourier--Sato transform of a cuspidal perverse sheaf is cuspidal, we see that $\bT_\fg i_* M$ is supported on $\cN_G$, so $j^* \bT_\fg i_*M = 0$.  Thus, $\cK(M) = 0$.

{\it Step 2. If $M = \indLG N$ for some $N \in \Perv_L(\cN_L,\bF)$, where $L$ is a proper Levi subgroup, then $\cK(M) = 0$.}  The second and third terms of~\eqref{eqn:func-dt} commute with $\indLG$, so we can find an (a priori non-canonical) isomorphism
\[
\cK(M) \cong \indLG \cK_L(N).
\]
By our induction hypothesis $\cK_L(N) = 0$, and thus $\cK(M) = 0$.

{\it Step 3. If $M$ is any simple perverse sheaf, then $\cK(M) = 0$.}  Assume that this is not the case, and consider the perverse cohomology objects $\pH^i(\cK(M))$.  Let $k(M)$ be the largest integer such that $\pH^{k(M)}(\cK(M)) \ne 0$.  Let $s$ be the maximum value of $k(M)$ as $M$ ranges over all simple perverse sheaves with $\cK(M) \ne 0$.  Thus, $\pH^{s
+1}(\cK(M)) = 0$ for all simple $M$, and it follows that
\[
\pH^{s+1}(\cK(N)) = 0
\qquad\text{for all $N \in \Perv_G(\cN_G)$,}
\]
not necessarily simple.  Now, take a simple $M$ such that $\cK(M) \ne 0$ and $k(M) = s$.  In view of Step~1, $M$ cannot be cuspidal, so by Lemma~\ref{lem:induced}, there is some surjective map $P \to M$ where $P$ is induced from a proper Levi subalgebra.  Let $Q$ be the kernel of $P \to M$.  By Step~2, $\cK(P) = 0$, so from the long exact sequence in perverse cohomology associated to $\cK(Q) \to \cK(P) \to \cK(M) \to$, we have
\[
\pH^i(\cK(M)) \cong \pH^{i+1}(\cK(Q))
\]
for all $i$.  Since the right-hand side vanishes for $i+1 = s+1$ but the left-hand side is nonzero for $i = s$, we have a contradiction, and $\cK(M) = 0$ for all simple perverse sheaves $M$.
\end{proof}

%%%%%%%%%%%%%%%%%%%%%%%%%%%%%%%%%%%%%%%%%%%%%%%%%%%%%%%%%%%%%%%%%%%%%%%%%%
\section{Skyscraper and Richardson sheaves}
\label{sect:richardson}
%%%%%%%%%%%%%%%%%%%%%%%%%%%%%%%%%%%%%%%%%%%%%%%%%%%%%%%%%%%%%%%%%%%%%%%%%%

In this section, we study the behavior of $\cR$ on skyscraper sheaves and on sheaves induced from skyscraper sheaves, known as \emph{Richardson sheaves}.  The terminology comes from the following observation: given a parabolic subgroup $P \subset G$, let $\pi: G \times^P \fu \to \fg$ be the map $\pi(g,x) = g \cdot x$.  It is easy to see from the definition that for any Levi subgroup $L \subset G$, we have
\begin{equation}\label{eqn:richardson}
\indLG \d_\fl \cong \pi_*\uF_{G \times^P \fu}[\dim G \times^P \fu].
\end{equation}
Thus, the support of $\indLG \d_\fl$ is the image of $\pi$, which is the closure of the Richardson orbit associated to $\fl$.

For the next proposition, let $\cO_\reg \subset \cN_G$ denote the regular nilpotent orbit.  Also, let $\Cent(G) \subset G$ denote the center of $G$, and let $\Cento(G) \subset \Cent(G)$ be its identity component. 

\begin{prop} \label{prop:const}
The object $\cR(\d_\fg)$ is a perverse sheaf.  In fact, we have
\[
\cR(\d_\fg) \cong \uF_{\cN_G}[\dim \cN_G] \cong \Delta(\cO_\reg,\uF).
\]
%where $\cO_\reg \subset \cN_G$ denotes the regular nilpotent orbit.
Moreover, if the characteristic of $\bF$ does not divide the order of $\Cent(G)/\Cento(G)$, then $\cR(\d_\fg)$ is a projective object in $\Perv_G(\cN_G,\bF)$.
\end{prop}
\begin{proof}
Following the definitions gives
\[
\cR \d = i^* \bT \d [\dim \cN_G - \dim \fg] = i^*
\uF_\fg [\dim \cN_G] = \uF_{\cN_G}[\dim \cN_G].
\]
Recall that $\cN$ is a complete intersection~\cite{kostant}, and
that on a complete intersection, the constant sheaf shifted by the
dimension is perverse.

Consider the adjunction triangle for the inclusion $j$ of the regular
orbit $\cO_\reg$ into $\cN_G$ and the perverse sheaf $\uF_{\cN_G}[\dim\cN_G]$,
\[ 
j_! \uF_{\cO_\reg}[\dim\cN_G] \to \uF_{\cN_G}[\dim\cN_G] \to
\uF_{\cN_G \smallsetminus \cO_\reg}[\dim\cN_G] \to
\]
As the codimension of $\cN_G \smallsetminus \cO_\reg$ is at least 2, $\uF_{\cN_G \smallsetminus \cO_\reg}[\dim\cN_G] \in {}^p\cD_G(\cN)^{\leq -2}$.  Taking the long exact sequence in perverse cohomology, we obtain
\[
\pH^0(j_!\uF_{\cO_\reg}[\dim\cN_G]) \cong \uF_{\cN_G}[\dim\cN_G],
\]
as desired.

Another consequence is that for any $M \in \Perv_G(\cN_G,\bF)$, we have an injective map $\Hom(\uF_{\cN_G}[\dim \cN_G], M[1]) \to \Hom(j_!\uF_{\cO_\reg}[\dim \cN_G], M[1])$.  By adjunction, this yields an injective map
\[
\Ext^1_{\Perv_G(\cN_G,\bF)}(\uF_{\cN_G}[\dim \cN_G], M) \to \Ext^1_{\Perv_G(\cO_\reg,\bF)}(\uF_{\cO_\reg}[\dim \cN_G], j^*M).
\]
Recall that the category of $G$-equivariant local systems on $\cO_\reg$ is equivalent to the category of representations of the finite group $\Cent(G)/\Cento(G)$.  Under the assumption that the characteristic of $\bF$ does not divide the order of this group, this category is semisimple, and the $\Ext^1$-groups above vanish for all $M$.  Thus, $\uF_{\cN_G}[\dim \cN_G]$ is projective.
\end{proof}

\begin{cor}\label{cor:proj}
For any Levi subgroup $L \subset G$, the object $\cR(\indLG \d_\fl)$ is perverse.  If the characteristic of $\bF$ does not divide the order of $\Cent(L)/\Cento(L)$, then $\cR(\indLG \d_\fl)$ is a projective object in $\Perv_G(\cN_G,\bF)$.
\end{cor}

%%%%%%%%%%%%%%%%%%%%%%%%%%%%%%%%%%%%%%%%%%%%%%%%%%%%%%%%%%%%%%%%%%%%%%%%%%%
\section{Ringel duality for $GL_n$}
\label{sect:ringel}
%%%%%%%%%%%%%%%%%%%%%%%%%%%%%%%%%%%%%%%%%%%%%%%%%%%%%%%%%%%%%%%%%%%%%%%%%%%

In this section, we restrict ourselves to the group $G = GL_n$.  Recall that for this group, the category $\Perv_G(\cN_G,\bF)$ is equivalent to the category of modules over the Schur algebra $S_\bF(n,n)$~\cite{mau}.  In particular, $\Perv_G(\cN_G,\bF)$ is quasi-hereditary and contains tilting and projective objects.  The aim of this section is to show that for $GL_n$, the functor $\cR$ is a geometric version of Ringel duality.

Let us first fix some notation.  We take $T$ (resp.~$B$) to be the subgroup of diagonal (resp.~upper triangular) matrices.  Recall that nilpotent orbits in $\cN_G$ are indexed by partitions of $n$.  Given a partition $\lambda = (\lambda_1 \ge \cdots \ge \lambda_k)$ of $n$, let $\cO_\lambda \subset \cN_G$ be the corresponding orbit.  Thus, $\cO_{(n)}$ is the regular orbit, and $\cO_{(1,\ldots,1)}$ is the zero orbit.  Let $L_\lambda \subset G$ be the Levi subgroup consisting of block-diagonal matrices with blocks of size $\lambda_1, \ldots, \lambda_k$, and let $P_\lambda$ be the parabolic containing $L_\lambda$ and $B$.  

Let $\lambda^\vee$ denote the dual partition to $\lambda$.  It is well known that the Richardson orbit associated to $L_\lambda$ is $\cO_{\lambda^\vee}$.  On the other hand, if $\cO^{L_\lambda}_\reg$ denotes the regular orbit in $\cN_{L_\lambda}$, we have
\begin{equation}\label{eqn:coind-orbit}
G \cdot \cO^{L_\lambda}_\reg = \cO_\lambda.
\end{equation}
For brevity, we write $\IC_\l$ for $\IC(\cO_\l,\uF)$.  Let $\cT_\l$ be an indecomposable tilting perverse sheaf with support $\overline{\cO_\l}$, and let $\cP_\l$ be a projective cover of $\IC_\l$.

\begin{thm} \label{thm:ringel}
Let $G=GL_n$. For any partition $\l$ of $n$, we have $\cR(\cT_{\l^\vee}) \cong \cP_\l$.
\end{thm}

For any partition $\l$, let $P_\l$ be the standard parabolic with diagonal blocks of size the parts of $\l$.  Let $\l^\vee$ denote the dual partition.

\begin{lem} \label{lem:tilt}
\begin{enumerate}
\item The induced perverse sheaf $\indLGl \d$ is tilting.
\item $\indLGl \d \cong \cT_{\l^\vee} \oplus \bigoplus_{\mu < \l^\vee} \cT_\mu^{m_\mu}$ where $m_\mu$ is some multiplicity.\label{it:tilt-supp}
\end{enumerate}
\end{lem}
\begin{proof}
(1)~Consider the map $\pi: G \times^{P_\l} \fu_\l \to \fg$.  According to~\eqref{eqn:richardson}, the object $M = \indLGl \d$ is isomorphic to $\pi_*\uF[\dim G \times^{P_\l} \fu_\l]$.  We begin by recalling that $M$ has cohomology concentrated in even degrees.  Indeed, by base change, the stalk of $M$ at a point $x \in \cO_\mu$ is (an even shift of) the cohomology of the fiber $\pi^{-1}(x)$.  By~\cite[Theorem~2.5]{bo}, $\pi^{-1}(x)$ has an affine paving, so $H^i(\pi^{-1}(x),\bF)$ vanishes for $i$ odd.

Next, we claim that the group
\begin{equation}\label{eqn:tilt-dev}
\Hom(\mathcal{H}^i(i_\mu^*M)[-i], \uF_{\cO_\mu}[\dim \cO_\mu+1])) \cong H_G^{i+\dim \cO_\mu +1}(\cO_\mu, \mathcal{H}^i(i_\mu^*M)^\vee)
\end{equation}
vanishes for all $i$.  (Here, $\mathcal{H}^i(i_\mu^*M)^\vee$ denotes the dual local system to $\mathcal{H}^i(i_\mu^*M)$.)  This is trivial if $i$ is odd, by the preceding paragraph.  On the other hand, if $i$ is even, recall that the $G$-equivariant cohomology of a homogeneous $G$-space with coefficients in any local system vanishes in odd degrees.

Finally, in any highest weight category, an object $X$ is tilting if and only if both $\Ext^1(\Delta_i,X)$ and $\Ext^1(X,\nabla_i)$ vanish for all $i$ (see~\cite[Lemma~4]{bez}).  In $\Perv_G(\cN_G,\bF)$, the group
\[
\Ext^1(M,\nabla_\mu)=\Hom(i_\mu^* M, \ul{\bF}_{\cO_\mu}[\dim \cO_\mu+1]).
\]
vanishes, by~\eqref{eqn:tilt-dev} and d\'evissage, and then $\Ext^1(\Delta_\mu,M) = 0$ by Verdier duality.

(2) It is well known that the image of the map $\pi: G \times^{P_\l} \fu_\l \to \fg$ is $\overline{\cO_{\l^\vee}}$, and that it is an isomorphism over $\cO_{\l^\vee}$.  We have already seen that $\indLGl \d$ is tilting, which implies that it is a direct sum of indecomposable tilting objects $\cT_\mu$.  Considering supports completes the argument.
\end{proof}

\begin{rmk}
In the preceding proof, we essentially repeated the argument of~\cite[\S 4.3.1]{jmw} showing that $\indLGl\d$ is a \emph{parity sheaf}.  Thus, this lemma shows that all parity sheaves on $\cN_G$ are tilting perverse sheaves.
\end{rmk}

We are now ready to prove Theorem \ref{thm:ringel}.

\begin{proof}
We proceed by induction with respect to the partial order on partitions.  For $\l = (n)$, we saw in Corollary~\ref{cor:proj} that $\cR (\cT_{(1^n)}) = \cR \d \cong \cP_{(n)}$.  Suppose now that we have shown that for all $\mu > \l$, $\cR (\cT_{\mu^\vee}) \cong \cP_{\mu}$.  We wish to check that $\cR (\cT_{\l^\vee}) \cong \cP_{\l}$.

As $\cR$ is an equivalence and $\cT_{\l^\vee}$ is indecomposable, $\cR(\cT_{\l^\vee})$ is also indecomposable.  By Lemma~\ref{lem:tilt}\eqref{it:tilt-supp}, $\cT_{\l^\vee}$ is a summand of $\indLGl \d$, so $\cR(\cT_{\l^\vee})$ is a summand of $\cR(\indLGl \d)$.  The latter is projective by Corollary~\ref{cor:proj},  so $\cR(\cT_{\l^\vee}) \cong \cP_\nu$ for some partition $\nu$.

It remains to show that $\cR (\indLGl \d)$ is a sum of indecomposable projectives $\cP_\mu$ for $\mu \geq \l$.  Equivalently, we must show that if $\mu \not\geq \l$, then there is no nonzero map $\indLGl \cR(\d) \to \IC_\mu$.  This follows from Lemmas~\ref{lem:ind-std} and~\ref{prop:const} together with~\eqref{eqn:coind-orbit}.
\end{proof}

As a corollary, we recover Donkin's result that the Schur algebra $S_\bF(n,n)$ is Ringel self-dual:

\begin{cor}\label{cor:donkin}
We have
$\End( \bigoplus_\lambda \cT_\lambda ) \cong \End( \bigoplus_\lambda \cP_\lambda)$.\qed
\end{cor}

Next, let $\Tilt_G(\cN_G,\bF)$ and $\Proj_G(\cN_G,\bF)$ denote the additive categories of tilting and projective perverse sheaves, respectively, on $\cN_G$.  Another immediate consequence of Theorem~\ref{thm:ringel} is that $\cR$ induces an equivalence of additive categories
\begin{equation}\label{eqn:gringel-add}
H_\cR: \Tilt_G(\cN_G,\bF) \simto \Proj_G(\cN_G,\bF).
\end{equation}
This equivalence plays the role of~\eqref{eqn:ringel-add}.  Note that even though the isomorphism of Corollary~\ref{cor:donkin} depends on noncanonical choices (because the proof of Theorem~\ref{thm:ringel} does), the functor $H_\cR$ does not.  This functor extends to an equivalence of homotopy categories, and hence, using~\eqref{eqn:tiltproj}, to an equivalence of derived categories
\[
H_\cR: \Db\Perv_G(\cN_G,\bF) \simto \Db\Perv_G(\cN_G,\bF).
\]

We conclude this section by explaining the way in which this autoequivalence is compatible with $\cR$.  Recall that even though $\Db\Perv_G(\cN_G,\bF)$ and $\cD_G(\cN_G,\bF)$ are not equivalent, there is a $t$-exact functor
\[
\real: \Db\Perv_G(\cN_G,\bF) \to \cD_G(\cN_G,\bF)
\]
whose restriction to $\Perv_G(\cN_G,\bF)$ is the identity functor.  

\begin{prop}\label{prop:ringel-compat}
Let $G = GL_n$.  The following diagram commutes:
\begin{equation}\label{eqn:ringel-compat}
\vcenter{\xymatrix{
\Db\Perv_G(\cN_G,\bF) \ar[rr]^-{\real} \ar[d]_{H_\cR} &&
  \cD_G(\cN_G,\bF) \ar[d]^{\cR} \\
\Db\Perv_G(\cN_G,\bF) \ar[rr]_-{\real} &&
  \cD_G(\cN_G,\bF)}}
\end{equation}
\end{prop}
\begin{proof}
We briefly review the construction of the realization functor.  Let $\tilde\cD_G(\cN_G,\bF)$ be the ``filtered equivariant derived category'' of $\cN$.  (This is defined using the methods of~\cite{bl}; it is related to the filtered derived category considered in~\cite[\S 3.1]{bbd} in the same way that $\cD_G(\cN_G,\bF)$ is related to the ordinary derived category of $\cN$.)  This category has a $t$-structure whose heart is equivalent to the abelian category $\Cb\Perv_G(\cN_G,\bF)$ of bounded chain complexes of objects in $\Perv_G(\cN_G,\bF)$.  There is also a ``filtration-forgetting'' functor $\omega: \tilde\cD_G(\cN_G,\bF) \to \cD_G(\cN_G,\bF)$.  The restriction
\[
\omega|_{\Cb\Perv_G(\cN_G,\bF)}: \Cb\Perv_G(\cN_G,\bF) \to \cD_G(\cN_G,\bF)
\]
factors through the natural functor $\Cb\Perv_G(\cN_G,\bF) \to \Db\Perv_G(\cN_G,\bF)$, and thus gives rise to the realization functor.

The usual sheaf functors all lift to the setting of filtered derived categories.  In particular, there is a functor $\tilde\cR: \tilde\cD_G(\cN_G,\bF) \to \tilde\cD_G(\cN_G,\bF)$ such that $\real \circ \tilde\cR \cong \cR \circ \real$.  In general, $\tilde\cR$ does not preserve the category $\Cb\Perv_G(\cN_G,\bF)$, but it does behave well on the additive subcategories 
\[
\Cb\Tilt_G(\cN_G,\bF)
\qquad\text{and}\qquad
\Cb\Proj_G(\cN_G,\bF)
\]
consisting of chain complexes of tilting and projective sheaves, respectively.  In particular, it follows from Theorem~\ref{thm:ringel} that we have a commutative diagram
\[
\xymatrix{
\Cb\Tilt_G(\cN_G,\bF) \ar[rr]^-{\omega} \ar[d]_{\tilde\cR|_{\Cb\Tilt_G(\cN_G,\bF)}} &&
  \cD_G(\cN_G,\bF) \ar[d]^{\cR} \\
\Cb\Proj_G(\cN_G,\bF) \ar[rr]_-{\omega} &&
  \cD_G(\cN_G,\bF)}
\]
The functor $\tilde\cR|_{\Cb\Tilt_G(\cN_G,\bF)}$ must coincide with the one induced by~\eqref{eqn:gringel-add}.  (To see this, recall that the terms and differentials of a complex in $\Cb\Perv_G(\cN_G,\bF) \subset \tilde\cD_G(\cN_G,\bF)$ can be described using ``baric truncation functors'' on $\tilde\cD_G(\cN_G,\bF)$, and then observe that $\tilde\cR$ commutes with these functors.)  Since null-homotopic maps in $\Cb\Perv_G(\cN_G,\bF)$ are sent to $0$ by $\omega$, this commutative diagram gives rise to a similar one in which the categories in the left-hand column are replaced by the homotopy categories $\Kb\Tilt_G(\cN_G,\bF)$ and $\Kb\Proj_G(\cN_G,\bF)$.  The desired diagram then follows from~\eqref{eqn:tiltproj}.
\end{proof}

%%%%%%%%%%%%%%%%%%%%%%%%%%%%%%%%%%%%%%%%%%%%%%%%%%%%%%%%%%%%%%%%%%%%%%%%%%%
\section{Behavior of $\cR$ in other groups}
\label{sect:behavior}
%%%%%%%%%%%%%%%%%%%%%%%%%%%%%%%%%%%%%%%%%%%%%%%%%%%%%%%%%%%%%%%%%%%%%%%%%%%

%--------------------------------------------------------------------------
\subsection{Characteristic zero}
%--------------------------------------------------------------------------

In this section, we once again allow $G$ to be an arbitrary reductive group, but we assume that $\bF = \C$.  In this case, there is a classification of simple perverse sheaves on $\cN_G$ in terms of representations of certain finite Coxeter groups, which we will now briefly review.  For a Levi subgroup $L \subset G$, let $\cW_L = N_G(L)/L$, where $N_G(L)$ denotes the normalizer of $L$ in $G$.  The \emph{generalized Springer correspondence}~\cite{lus:icc} is a bijection
\begin{equation}\label{eqn:genspringer}
\tilde\nu:
\coprod_{\substack{\{(L,S)\} \\ S \in \Perv_L(\cN_L,\C) \\ \text{$S$ simple, cuspidal}}} \Irr(\cW_L) \simto
\left\{
\begin{array}{c}
\text{isomorphism classes of simple} \\
\text{objects in $\Perv_G(\cN_G,\C)$}
\end{array}
\right\},
\end{equation}
where the pairs $(L,S)$ indexing the disjoint union are considered only up to $G$-conjugacy.  The pair $(T,S_0)$ (where $S_0$ is the unique simple perverse sheaf on the point $\cN_T$) always occurs in the left-hand member.  The restriction of $\nu$ to $\Irr(\cW_T)$ is the ``ordinary'' Springer correspondence.

Given a pair $(L,S)$ as above and an element $\chi \in \Irr(\cW_L)$, choose a representative simple perverse sheaf $\IC_{L,S,\chi}$.  The following statement, which describes the behavior of $\cR$ in terms of the classification~\eqref{eqn:genspringer}, is a restatement of~\cite[\S5.5]{lus:usir}.  For perverse sheaves appearing in the ordinary Springer correspondence, this statement was known much earlier; see~\cite[Theorem~5.2]{hk} and~\cite[Proposition~17.7]{shoji}.

\begin{prop}\label{prop:genspring}
For each pair $(L,S)$ as in~\eqref{eqn:genspringer} and each $\chi \in \Irr(\cW_L)$, we have $\cR(\IC_{L,S,\chi}) \cong \IC_{L,S,\chi \otimes \e}[\dim \cN_L - \dim [\fl,\fl]]$. \qed
\end{prop}

Let us know restrict our attention to the Serre subcategory $\Spr \subset \Perv_G(\cN_G,\C)$ generated by the simple perverse sheaves in the ordinary Springer correspondence.  Note that $\dim \cN_T = \dim [\ft,\ft] = 0$, so this category is preserved by $\cR$.  On $\Spr$, Proposition~\ref{prop:genspring} can be understood in a categorical way, as a description of the autoequivalence of $\Spr$ induced by $\cR$.  One might seek to generalize this to the triangulated category $\cD_\Spr \subset \cD_G(\cN_G,\C)$ generated by $\Spr$.

In recent work~\cite{rider}, Rider has shown how to describe $\cD_\Spr$ in a way that extends the ordinary Springer correspondence. Consider the smash product $E = \C[W] \mathbin{\#} \C[\ft^*]$ as a graded ring concentrated in nonpositive even degrees.  More precisely, we regard $E$ as a differential-graded ring with zero differential.  Rider's result states that there is an equivalence of triangulated categories
\begin{equation}\label{eqn:rider}
\Theta: \cD_\Spr \simto DG(E),
\end{equation}
where $DG(E)$ denotes the derived category of finitely-generated dg-modules over $E$.  Note that it makes sense to take the tensor product of a dg-$E$-module with a finite-dimensional $W$-representation.  In particular, tensor product with $\e$ is a well-defined autoequivalence of $DG(E)$.  The following conjecture is a derived version of the restriction to $\Spr$ of Proposition~\ref{prop:genspring}.

\begin{conj}\label{conj:rider}
There is an isomorphism of functors $\cR|_{\cD_\Spr} \cong \Theta^{-1} \circ ({-} \otimes \e) \circ \Theta$.
\end{conj} 

\begin{rmk}
\begin{enumerate}
\item On the additive subcategory of $\cD_\Spr$ consisting of semisimple objects (i.e., direct sums of shifts of simple perverse sheaves), this statement can be deduced from~\cite[Theorem~3.8]{em}.

\item The analogous derived-category statement in the setting of mixed perverse sheaves on a Lie algebra over a finite field can likely be proved using the method of ``Orlov categories'' introduced in~\cite{ar}.
\end{enumerate}
\end{rmk}

If~\eqref{eqn:rider} has analogues for the other ``blocks'' of the generalized Springer correspondence, one could formulate corresponding analogues of Conjecture~\ref{conj:rider}, which together would give a complete description of $\cR$ in characteristic zero.

%--------------------------------------------------------------------------
\subsection{Examples in rank two}  
%--------------------------------------------------------------------------

We conclude by presenting the results of some calculations in types $B_2$ and $G_2$, using the language of parity sheaves~\cite{jmw}.  Given an orbit $\cO_\alpha \subset \cN_G$, we will denote by $\cE_\alpha$ an indecomposable parity extension of $\uF_{\cO_\alpha}[\dim \cO_\alpha]$.  Similarly, if $\cL$ is some nontrivial local system on some nilpotent orbit $\cO$, then $\cE_\cL$ will denote an indecomposable parity extension of $\cL[\dim \cO]$.  

\begin{rmk}
Some care is required when working with parity sheaves on $\cN_G$ in small characteristics, because the parity condition~\cite[Equation~(2.2)]{jmw} can fail, and so the general uniqueness result for parity sheaves~\cite[Theorem~2.9]{jmw} is not necessarily available.  The notation above will be used for certain explicitly-constructed objects, so there is no risk of ambiguity.  Most of the parity sheaves we consider happen to be perverse; it is possible that a weaker form of~\cite[Theorem~2.9]{jmw} holds for such objects.
\end{rmk}

It can be shown by a general argument (cf.~the remarks following~\cite[Theorem~3.2.1]{bgs}) that $\Perv_G(\cN_G)$ has enough projectives.  We will let $P_\alpha$ (resp.~$P_\cL$) denote the projective cover of the simple perverse sheaf $\IC(\cO_\alpha,\uF)$ (resp.~$\IC(\cO,\cL)$ where $\cL$ is an irreducible local system on $\cO$).

We focus on computing $\cR$ on direct summands of Richardson sheaves, and especially on the object $\ind_T^G \delta$, known as the \emph{Springer sheaf}.

%..........................................................................
\subsubsection{Type $B_2$}  
%..........................................................................

For $G=SO_5$, the nilpotent cone has four orbits: the regular~$\cO_{\reg}$, the subregular~$\cO_{\sub}$, the minimal~$\cO_{\min}$, and the zero orbit~$\cO_0$.  The equivariant fundamental groups of the orbit are trivial, except for the subregular, which has equivariant fundamental group $\Z/2$.  Thus, the category of perverse sheaves $\Perv_G(\cN_G,\bF)$ has four simple objects when the characteristic of $\bF$ is equal to 2, and has five simple objects in all other characteristics.

Let $\e$ denote the sign representation of $\Z/2$ and $\tilde{\mathbf{1}}$ denote the 2-dimensional projective cover of the trivial representation of $\Z/2$ in characteristic 2.  We will use these to label the corresponding $G$-equivariant local systems on the subregular orbit.

The Weyl group $W$ is the dihedral group of order 8 and has 5 conjugacy classes, only one of which is 2-regular.

\medskip

It follows that in characteristic greater than 2, the Springer sheaf contains each simple perverse sheaf as a direct summand.  The $\IC$-sheaves are the parity sheaves $\cE_{\reg}$, $\cE_{\sub}$, $\cE_{\e}$, $\cE_{\min}$, $\cE_0$, and the category $\Perv_G(\cN_G,\bF)$ is semisimple.  In this case, the functor $\cR$ will permute the $\IC$-sheaves as in characteristic zero.

\medskip

Let $\bF$ have characteristic 2.  In this case the regular representation of $W$ is indecomposable and therefore the Springer sheaf is too.  Similarly, one can show that the Richardson sheaves from the two nontrivial Levi subgroups are irreducible and are parity extensions of the trivial local system and its projective cover $\tilde{\mathbf{1}}$ on the subregular orbit.  Thus, as summands of the various Richardson sheaves (including the skyscraper sheaf $\delta_\fg$), we obtain four indecomposable parity complexes $\cE_{\reg}$, $\cE_{\sub}$, $\cE_{\tilde{\mathbf{1}}}$, $\cE_0$.  By Theorem~\ref{thm:main} and Corollary~\ref{cor:proj}, the functor $\cR$ takes each of these to the four indecomposable projective objects in $\Perv_G(\cN_G,\bF)$.  Moreover, one can show that the correspondence is as follows:
\begin{align*}
\cR(\cE_{\reg}) &\cong \cE_{\reg} \cong P_0 \\
\cR(\cE_{\sub}) &\cong P_{\sub} \\
\cR(\cE_{\tilde{\mathbf{1}}}) &\cong P_{\min} \\
\cR(\cE_0) &\cong \Delta_{\reg} \cong P_{\reg}
\end{align*}

%..........................................................................
\subsubsection{Type $G_2$}
%..........................................................................

For $G$ of type $G_2$, the nilpotent cone has five orbits: the regular $\cO_{\reg}$, the subregular $\cO_{\sub}$, the middle $\cO_{\midd}$, the minimal $\cO_{\min}$,  and the zero orbit $\cO_0$.  The equivariant fundamental groups of the orbit are trivial, except for the subregular, which has equivariant fundamental group $S_3$.  Thus, the category of perverse sheaves $\Perv_G(\cN_G,\bF)$ has six simple objects when the characteristic of $\bF$ is equal to 2 or 3, and has seven simple objects in all other characteristics.

Let $\e$ denote the sign representation of $S_3$ and $\sigma$ the simple 2-dimensional representation in characteristic not equal to 3.  Let $\tilde{\mathbf{1}}$ denote the 3-dimensional projective cover of the trivial representation of $S_3$ in characteristic 2.  As before, we will use these to label the corresponding $G$-equivariant local systems on the subregular orbit.

The Weyl group $W$ is the dihedral group of order 12 and has 6 conjugacy classes, 4 of which are 3-regular and 2 of which are 2-regular.

\medskip

In characteristic not equal to 2 or 3, the Springer sheaf is semisimple and every $\IC$-sheaf appears with the exception of $\IC(\cO_{\sub},\e)$, which is cuspidal and clean.  Then the category of perverse sheaves is semisimple and the functor $\cR$ acts on the $\IC$-sheaves just as in characteristic zero.

\medskip

When the characteristic of $\bF$ is 2, we find that the Richardson sheaves break up as follows:
\begin{align*}
\ind^G_T \d &\cong \cE_{\reg} \oplus \cE_{\sub}^{\oplus 2}, \\
\ind^G_{L_1} \d &\cong \cE_{\sub} \oplus \cE_{\sigma}, \\
\ind^G_{L_2} \d &\cong \cE_{\sub} \oplus \cE_{\min}
\end{align*}
and of course $\ind^G_G \d \cong \d$.  In this way we have obtained five indecomposable summands which are therefore taken to five indecomposable projective perverse sheaves.  One can show that the correspondence is as follows:
\begin{align*}
\cR(\cE_{\reg}) &\cong \cE_{\reg} \cong P_0 \\
\cR(\cE_{\sigma}) &\cong P_{\min} \\
\cR(\cE_{\sub}) &\cong \cE_{\sub} \cong P_{\midd} \\
\cR(\cE_{\min}) &\cong P_\sigma \\
\cR(\cE_0) &\cong \Delta_{\reg} \cong P_{\reg}
\end{align*}
Thus the only projective cover not obtained in this way is $P_{\sub}$.

\medskip

When the characteristic of $\bF$ is 3, we find that the Richardson sheaves break up as follows:
\begin{align*}
\ind^G_T \d &\cong \cE_{\reg} \oplus \cE_{\sub} \oplus \cE_{\tilde{\mathbf{1}}} \oplus \cE_{\midd}, \\
\ind^G_{L_1} \d &\cong \cE_{\tilde{\mathbf{1}}} \oplus \cE_{\midd}, \\
\ind^G_{L_2} \d &\cong \cE_{\sub} \oplus \cE_{\midd}
\end{align*}
and of course $\ind^G_G \d \cong \d$.  In this way we again obtained 5 indecomposable summands which are therefore taken to 5 indecomposable projective perverse sheaves.  One can show that the correspondence is as follows:
\begin{align*}
\cR(\cE_{\reg}) &\cong \cE_{\midd} \cong P_0 \\
\cR(\cE_{\tilde{\mathbf{1}}}) &\cong \cE_{\sub} \cong P_{\min} \\
\cR(\cE_{\midd}) &\cong \cE_{\reg} \cong P_{\midd} \\
\cR(\cE_{\sub}) &\cong \cE_{\tilde{\mathbf{1}}} \cong P_{\sub} \\
\cR(\cE_0) &\cong \Delta_{\reg} \cong P_{\reg}
\end{align*}
Thus the only projective cover not obtained in this way is $P_{\e}$.  It is interesting to note that in both this case and the preceding one, the unique ``missing'' projective is the projective cover of the $\IC$-sheaf corresponding to the local system obtained by modular reduction of the unique cuspidal local system in characteristic zero.


\begin{thebibliography}{JMW}

\bibitem[AHR]{ahr}
P.~Achar, A.~Henderson, and S.~Riche, {\em Geometric Satake, Springer correspondence, and small representations II}, arXiv:1205.5089.

\bibitem[AR]{ar}
P.~Achar and S.~Riche, {\em Koszul duality and semisimplicity of Frobenius},
  Ann. Inst. Fourier, to appear.

\bibitem[BBD]{bbd}
A.~Be{\u\i}linson, J.~Bernstein, and P.~Deligne, {\em Faisceaux pervers},
  Analyse et topologie sur les espaces singuliers, I (Luminy, 1981),
  Ast\'erisque, vol. 100, Soc. Math. France, Paris, 1982, pp.~5--171.

%\bibitem[BBM]{bbm}
%A.~Be{\u\i}linson, R.~Bezrukavnikov, and I.~Mirkovi{\'c}, {\em Tilting
%  exercises}, Moscow Math. J. {\bf 4} (2004), 547--557.

\bibitem[BGS]{bgs}
A.~Beilinson, V.~Ginzburg, W.~Soergel,
{\em Koszul duality patterns in representation theory},
J. Amer. Math. Soc. {\bf 9} (1996), no. 2, 473Ð527. 


\bibitem[BL]{bl}
J.~Bernstein and V.~Lunts, {\em Equivariant sheaves and functors}, Lecture
  Notes in Mathematics, vol. 1578, Springer-Verlag, Berlin, 1994.

\bibitem[Bez]{bez}
R.~Bezrukavnikov, {\em Cohomology of tilting modules over quantum groups and
  {$t$}-structures on derived categories of coherent sheaves}, Invent. Math.
  {\bf 166} (2006), 327--357.

%\bibitem[BM]{bm:rgw}
%W.~Borho and R.~MacPherson, {\em Repr\'esentations des groupes de {W}eyl et
%  homologie d'intersection pour les vari\'et\'es nilpotentes}, C. R. Acad. Sci.
%  Paris S\'er. I Math. {\bf 292} (1981), 707--710.

\bibitem[BO]{bo}
J.~Brundan and V.~Ostrik, {\em Cohomology of Spaltenstein varieties},
  Transform. Groups {\bf 16} (2011), 619--648.

\bibitem[Br]{bra}
T.~Braden, {\em Hyperbolic localization of intersection cohomology}, Transform.
  Groups {\bf 8} (2003), 209--216.

\bibitem[CPS]{cps} 
E.~Cline, B.~Parshall, L.~Scott,
{\em Finite-dimensional algebras and highest weight categories}, 
J. Reine Angew. Math. {\bf 391} (1988), 85Ð99. 

\bibitem[Do]{do}
S.~Donkin, {\em On tilting modules for algebraic groups}, Math. Z. {\bf 212} (1993), 39--60.

\bibitem[EM]{em}
S.~Evens and I.~Mirkovi\'{c}, {\em Fourier transform and the Iwahori--Matsumoto
  involution}, Duke Math. J. {\bf 86} (1997), 435--464.

\bibitem[Ha]{happel}
D.~Happel, {\em On the derived category of a finite-dimensional algebra},
  Comment. Math. Helv. {\bf 62} (1987), 339--389.

\bibitem[HK]{hk}
R.~Hotta and M.~Kashiwara, {\em The invariant holonomic system on a semisimple
  Lie algebra}, Invent. Math. {\bf 75} (1984), 327--358.

\bibitem[J]{juteau}
D.~Juteau, {\em Modular Springer correspondence and decomposition matrices},
  Ph.D. thesis, Universit\'e Paris 7, 2007.

\bibitem[JMW]{jmw}
D.~Juteau, C.~Mautner, and G.~Williamson, {\em Parity sheaves}, preprint,
  arXiv:0906.2994.

\bibitem[KS]{ks}
M.~Kashiwara and P.~Schapira, {\em Sheaves on manifolds}, Grundlehren der
  mathematischen Wissenschaften, vol. 292, Springer-Verlag, Berlin, 1994.

\bibitem[Ko]{kostant}
B.~Kostant, {\em Lie group representations on polynomial rings}, Amer. J. Math.
  {\bf 85} (1963), 327--404.

\bibitem[La]{laumon}
G.~Laumon, {\em Transformation de Fourier, constantes d'\'{e}quations
  fonctionelles et conjecture de Weil}, Inst. Hautes \'Etudes Sci. Publ. Math.
  (1987), no.~65, 131--210.

%\bibitem[L1]{lus:gpsuc}
%G.~Lusztig, {\em Green polynomials and singularities of unipotent classes},
%  Adv. in Math. {\bf 42} (1981), 169--178.

\bibitem[L1]{lus:icc}
G.~Lusztig, {\em Intersection cohomology complexes on a reductive group},
  Invent. Math. {\bf 75} (1984), 205--272.

\bibitem[L2]{lus:ftsla}
G.~Lusztig, {\em Fourier transforms on a semisimple Lie algebra over {${\bf
  F}_q$}}, Algebraic groups Utrecht 1986, Lecture Notes in Mathematics, vol.
  1271, Springer, Berlin, 1987, pp.~177--188.

\bibitem[L3]{lus:usir}
G.~Lusztig, {\em A unipotent support for irreducible representations}, Adv.
  Math. {\bf 94} (1992), 139--179.

\bibitem[M]{mau}
C.~Mautner, {\em Sheaf theoretic methods in modular representation theory},
  Ph.D. thesis, University of Texas at Austin, 2010.

\bibitem[Mi]{mirk}
I.~Mirkovi\'{c}, {\em Character sheaves on reductive Lie algebras}, Mosc. Math. J. {\bf
  4} (2004), 897--910, 981.

\bibitem[PS]{ps}
B.~Parshall and L.~Scott, {\em Derived categories, quasi-hereditary algebras,
  and algebraic groups}, Proceedings of the Ottawa-Moosonee Conference,
  Carleton University, 1988, pp.~1--105.

\bibitem[Rid]{rider}
L.~Rider, {\em Formality for the nilpotent cone and a derived Springer correspondence}, arXiv:1206.4343.

\bibitem[Rin]{ringel}
C.~M. Ringel, {\em The category of modules with good filtrations over a
  quasi-hereditary algebra has almost split sequences}, Math. Z. {\bf 208}
  (1991), 209--223.

\bibitem[Sh]{shoji}
T.~Shoji, {\em Geometry of orbits and Springer correspondence}, Orbites
  unipotentes et repr\'{e}sentations, I, Ast\'erisque, vol. 168, 1988,
  pp.~61--140.

%\bibitem[Sp]{spalt}
%N.~Spaltenstein, {\em The fixed point set of a unipotent transformation on the
%  flag manifold}, Nederl. Akad. Wetensch. Proc. Ser. A {\bf 79} (1976),
%  452--456.

\end{thebibliography}
\end{document}